\documentclass[12pt,reqno]{amsart}

\usepackage[top=25mm, bottom=25mm, left=30mm, right=30mm]{geometry}
\usepackage{mathptmx}
\usepackage{mathrsfs}
\usepackage{url}
\usepackage{amsmath}
\usepackage{amssymb}
\usepackage[colorlinks=true,citecolor=red,linkcolor=red,urlcolor=blue]{hyperref}
\hypersetup{
  pdftitle={Adherence Semigroups and Density Finite-Sums Configurations},
  pdfauthor={Song Shao and Hui Xu}
}

\newtheorem{thm}{Theorem}[section]
\newtheorem{lem}[thm]{Lemma}

\newtheorem{prop}[thm]{Proposition}
\newtheorem{cor}[thm]{Corollary}

\theoremstyle{definition}

\theoremstyle{remark}
\newtheorem{rem}[thm]{Remark}

\numberwithin{equation}{section}

\newcommand{\N}{\mathbb N}
\newcommand{\Z}{\mathbb Z}
\newcommand{\A}{\mathcal A}

\begin{document}
\title{Adherence Semigroups and Density Finite-Sums Configurations}
\author[Song Shao]{Song Shao}

\address[S. Shao]{School of Mathematical Sciences, University of Science and Technology of China, Hefei, Anhui, 230026, P.R. China}

\email{songshao@ustc.edu.cn}

\author[H.~Xu]{Hui Xu}

\address[H. Xu]{Department of Mathematics, Shanghai Normal University, Shanghai, 200234, P.R. China}
\email{huixu@shnu.edu.cn}

\begin{abstract}
In this paper, we use the adherence semigroup to describe finite-sums
configurations in topological dynamical systems. We establish a correspondence
between exact finite sumsets and powers of adherence elements. As applications,
we give dynamical formulations of density finite-sums results, characterize
total minimality through the density of adherence-power orbits, and give a
positive answer to the ultrafilter question in
\cite[Question~8.9]{KMRR2024-1}. We also characterize the sets that belong to
the common sum of two commuting nonprincipal ultrafilters.
\end{abstract}


\maketitle

\section{Introduction}

Throughout, $\N=\{1,2,\ldots\}$ and $\Z_+=\{0,1,2,\ldots\}$.  For
$C\subseteq\N$ and $t\in\Z_+$, write
\[
 C-t:=\{n\in\N:n+t\in C\}.
\]

For a subset $B\subseteq\N$ and $k\in\N$, let
\[
 B^{\oplus k}:=\left\{\sum_{n\in F}n:F\subseteq B,\ |F|=k\right\}.
\]
The study of infinite sumsets in sets of positive density originates in questions of Erd\H{o}s. Moreira, Richter and Robertson proved that every set of positive upper Banach density contains $B+C$ for some infinite sets $B,C\subseteq\N$ \cite{MRR}. Kra, Moreira, Richter and Robertson later proved the $B+B+t$ conjecture \cite{KMRR2024-2}, the existence of sums $B_1+\cdots+B_k$ of arbitrarily many infinite sets \cite{KMRR2024-1}, and finally the density finite sums theorem \cite{KMRR2026}: for every $k$, a shift of the set contains $B^{\oplus j}$ simultaneously for $1\leq j\leq k$. For the development of these problems and further questions, see \cite{KMRRProblems2025}.

The dynamical approach in these works converts finite-sums configurations into
Erd\H{o}s progressions. In this paper, we give an algebraic description of
these progressions by means of the adherence semigroup. Our purpose is not to
reprove the density finite sums theorem, but to isolate the compact-semigroup
mechanism behind the passage between adherence powers and exact finite
sumsets. This also gives a positive answer, for sets of positive upper Banach
density, to the question on commuting nonprincipal ultrafilters posed in
\cite[Question~5.24]{KMRRProblems2025} and
\cite[Question~8.9]{KMRR2024-1}. The authors of
\cite{KMRRProblems2025} observed that the answer is positive for piecewise
syndetic sets, but left the positive-density case open.

By a topological system we mean a pair $(X,T)$, where $X$ is a compact
metric space and $T$ is a homeomorphism of $X$.  For $x\in X$ and
$E\subseteq X$, set
\[
 N(x,E):=\{n\in\N:T^nx\in E\}.
\]
The \emph{adherence semigroup} of $(X,T)$ is
\[
 \A(X,T)=\bigcap_{m=1}^{\infty}\overline{\{T^n:n\geq m\}},
\]
where the closure is taken in $X^X$ under the topology of pointwise
convergence.
 
The following theorem is the basic finite sums--adherence powers
correspondence.
\begin{thm}\label{main1}
Let $(X,T)$ be a topological dynamical system, let $x\in X$, and let
$k\in\N$. For any clopen sets $E_1,\ldots,E_k\subseteq X$, the following
are equivalent:
\begin{enumerate}
\item there exists $p\in\A(X,T)$ such that
\[
 (px,p^2x,\ldots,p^kx)\in E_1\times E_2\times\cdots\times E_k;
\]
\item there exists an infinite subset $B\subseteq\N$ such that
$B^{\oplus j}\subseteq N(x,E_j)$ for every $j=1,\ldots,k$.
\end{enumerate}
\end{thm}
This clopen version is particularly convenient in symbolic applications.
More general target sets are treated in Section~\ref{sec:dictionary}.

Combining Theorem~\ref{main1} with the density finite sums theorem gives the
following affirmative result for sets of positive upper Banach density.

\medskip

\noindent\textbf{Question}.\ \cite[Question 8.9]{KMRR2024-1}
{\em  Is it true that for any $A\subseteq \N$ with positive upper Banach density there exist
non-principal ultrafilters $\mathfrak{p}$ and $\mathfrak{q}$ with $\mathfrak{p} +\mathfrak{q} =\mathfrak{q} + \mathfrak{p}$ containing $A$?}

\medskip

We give a positive answer to this question.

\begin{thm}\label{main2}
If $A\subseteq\N$ has positive upper Banach density, then there exist
nonprincipal ultrafilters $\mathfrak p,\mathfrak q\in\beta\N$ such that
\[
A\in \mathfrak p+\mathfrak q=\mathfrak q+\mathfrak p.
\]
\end{thm}

We also study the following related question.

\medskip

\noindent\textbf{Question}.\ \cite[Question 5.24]{KMRRProblems2025}
{\em  For which sets $A\subseteq \N$ does one have
$A\in\mathfrak p+\mathfrak q=\mathfrak q+\mathfrak p$
for some pair $\mathfrak p,\mathfrak q$ of nonprincipal ultrafilters?}

\medskip

We obtain the following characterization.
\begin{thm}
\label{main3}
For $A\subseteq\N$, the following statements are equivalent:
\begin{enumerate}
\item there exist an infinite set $B\subseteq\N$ and $t\in\Z_+$ such that
      \[
       B^{\oplus2}\subseteq A-t;
      \]

\item there exist nonprincipal ultrafilters $\mathfrak p$ and
      $\mathfrak q$ such that
      \[
       A\in \mathfrak p+\mathfrak q=\mathfrak q+\mathfrak p.
      \]
\end{enumerate}
\end{thm}

Additional dynamical and square-ultrafilter conditions equivalent to the two
statements are given in Proposition~\ref{prop:dynamical-commuting-pair}. We remark a similar characterization of sumset property by nonprincipal ultrafilter is given in \cite[Proposition 3.1]{ACG} and \cite[Lemma 2.1]{MRR}, where they show that for   $A\subset \N$, there are nonprincipal filters $\mathfrak p$ and $\mathfrak q$ such tat $A\in (\mathfrak{p}+\mathfrak{q})\cap (\mathfrak{q}+\mathfrak{p})$ if and only if there are infinite subsets $B,C\subset \N$ such that $B+C\subset A$.

\subsection*{Organization of the paper}

Section~2 collects the required facts about adherence semigroups, density,
and recurrence.  In Section~\ref{sec:dictionary} we prove the
finite-sums--adherence-powers correspondence and compare it with
Erd\H{o}s progressions. Section~4 applies the correspondence to commuting
ultrafilters, powers of minimal systems, and all-order finite-sums
configurations.

\section{Preliminaries}

In this section, we recall the concepts and results used in the paper.

\subsection{The adherence semigroup}
Let $(X,T)$ be a topological dynamical system, and let
\[
 E(X,T)=\overline{\{T^n:n\in\Z_+\}}\subseteq X^X
\]
be the enveloping semigroup.  We use the multiplication convention
\[
 (pq)x=p(qx),\qquad p,q\in E(X,T),\ x\in X.
\]
Define the \emph{adherence semigroup} by
\begin{equation}\label{eq:adherence}
 \A(X,T)=\bigcap_{m=1}^{\infty}\overline{\{T^n:n\geq m\}}.
\end{equation}
All closures in $X^X$ are taken in the topology of pointwise convergence.
The set $\A(X,T)$ is a nonempty compact subsemigroup of $E(X,T)$, and
right multiplication $p\mapsto pq$ is continuous for each fixed
$q\in E(X,T)$; see \cite[Proposition~3.1]{BGKM2002}. Every
$p\in E(X,T)$ commutes with
$T$, although two elements of $E(X,T)$ need not commute.

For $x\in X$, write
\[
 \omega(x,T)=\bigcap_{m\geq1}\overline{\{T^nx:n\geq m\}}.
\]
Then
\[
 \omega(x,T)=\{px:p\in\A(X,T)\}.
\]
Indeed, one inclusion follows by evaluation.  For the other, take a net $(T^{n_\alpha}x)$ converging to $y\in\omega(x,T)$ and then pass, by compactness of $X^X$, to a subnet of the corresponding transformations.

\begin{lem}\label{lem-nets}\label{lem:escaping-net}
If $p\in\A(X,T)$, then there is a net $(n_\alpha)_{\alpha\in D}$ in
$\N$ such that
\[
 n_\alpha\longrightarrow\infty
 \quad\text{and}\quad
 T^{n_\alpha}\longrightarrow p
\]
pointwise on $X$.  Here $n_\alpha\to\infty$ means that for every
$M\in\N$ one has $n_\alpha\geq M$ eventually.
\end{lem}

\begin{proof}
Let $\mathfrak N(p)$ be the neighborhood system of $p$.  Direct
$\mathfrak N(p)\times\N$ by
\[
 (U,m)\preceq(V,\ell)
 \quad\Longleftrightarrow\quad
 V\subseteq U\ \text{and}\ \ell\geq m.
\]
For each $(U,m)$, equation \eqref{eq:adherence} permits us to choose
$n_{(U,m)}\geq m$ with $T^{n_{(U,m)}}\in U$.  The resulting net
converges pointwise to $p$, and its integer values tend to infinity.
\end{proof}

Consequently, if $y=px$ with $p\in\A(X,T)$ and $U$ is a neighborhood
of $y$, then  $n_{\alpha}\in N(x,U)$ enventually.  Its
tail has unbounded, and hence infinitely many, integer values.  Thus
$N(x,U)$ contains an infinite subset.

For an infinite $B\subseteq\N$, define
\begin{equation}\label{eq:AB}
 \A_B(X,T):=\bigcap_{m=1}^{\infty}
 \overline{\{T^b:b\in B,\ b\geq m\}}.
\end{equation}
The sets in this intersection are nonempty, compact and decreasing, so
$\A_B(X,T)$ is nonempty and is contained in $\A(X,T)$.  The same
neighborhood construction proves the useful equivalence
\begin{equation}\label{eq:AB-net}
 p\in\A_B(X,T)
 \quad\Longleftrightarrow\quad
 \text{there is a net }b_\alpha\in B\text{ with }
 b_\alpha\to\infty\text{ and }T^{b_\alpha}\to p.
\end{equation}
For the forward implication, repeat the proof of
Lemma~\ref{lem:escaping-net} with $\{T^b:b\in B,\ b\geq m\}$ in place
of $\{T^n:n\geq m\}$.  The converse follows because a net in
$B$ tending to infinity is eventually contained in every tail appearing in
\eqref{eq:AB}.

\subsection{Notions of subsets of integers}
The \emph{upper Banach density} of $A\subseteq\N$ is
\[
 d^*(A):=\limsup_{N\to\infty}\sup_{M\in\Z_+}
 \frac{|A\cap\{M+1,\ldots,M+N\}|}{N}.
\]
A \emph{F{\o}lner sequence in $\N$} is a sequence
$\Phi=(\Phi_N)_{N\in\N}$ of nonempty finite subsets of $\N$ such that
\[
 \lim_{N\to\infty}
 \frac{|(\Phi_N-h)\cap\Phi_N|}{|\Phi_N|}=1
 \qquad\text{for every }h\in\N.
\]
We use the standard equivalent characterization that
\begin{equation}\label{eq:banach-folner}
 d^*(A)=\sup_\Phi\limsup_{N\to\infty}
 \frac{|A\cap\Phi_N|}{|\Phi_N|},
\end{equation}
where the supremum ranges over F{\o}lner sequences in $\N$; see
\cite[Section~2]{KMRR2024-1}.

\begin{thm}[Density finite sums theorem \cite{KMRR2026}]
\label{thm:density-fs}
Let $A\subseteq\N$ satisfy $d^*(A)>0$, and let $k\in\N$. There exist an
infinite set $B\subseteq\N$ and $t\in\Z_{+}$ such that
\[
 \bigcup_{j=1}^k B^{\oplus j}\subseteq A-t.
\]
\end{thm}

In particular, for every fixed $k$ there are $B$ and $t$ with
$B^{\oplus k}\subseteq A-t$.  The shift cannot in general be removed, as
is already seen by taking $A$ to be the odd natural numbers.

By a \emph{system} we mean a triple $(X,\mu,T)$, where $(X,T)$ is a
topological system and $\mu$ is a $T$-invariant Borel probability measure.
A point $a\in X$ is \emph{generic for $\mu$ along $\Phi$}, written
$a\in\operatorname{gen}(\mu,\Phi)$, if
\[
 \mu=\lim_{N\to\infty}\frac{1}{|\Phi_N|}
 \sum_{n\in\Phi_N}\delta_{T^na}
\]
in the weak-star topology.

\begin{lem}\label{lem:generic-return-density}
Let $(X,\mu,T)$ be a system, let
$a\in\operatorname{gen}(\mu,\Phi)$, and let $U\subseteq X$ be open.  If
$\mu(U)>0$, then $d^*(N(a,U))>0$.
\end{lem}

\begin{proof}
By regularity, there is a continuous function $f\colon X\to[0,1]$ with
$\operatorname{supp}(f)\subseteq U$ and $\int f\,d\mu>0$.  Genericity gives
\[
 \lim_{N\to\infty}\frac{1}{|\Phi_N|}
 \sum_{n\in\Phi_N}f(T^na)=\int f\,d\mu>0.
\]
Since $f(T^na)\leq 1_{N(a,U)}(n)$, equation
\eqref{eq:banach-folner} yields the claim.
\end{proof}

A bridge between combinatorial problems and dynamical systems is the
following form of Furstenberg's correspondence principle
\cite[Theorem~2.10]{KMRR2024-1}.
\begin{thm}[Correspondence principle]
\label{thm-Fur-Prin}\label{thm:correspondence}
Let $A\subseteq\N$, and suppose $\Phi$ is a F{\o}lner sequence in $\N$
such that
\[
 \delta=\lim_{N\to\infty}\frac{|A\cap\Phi_N|}{|\Phi_N|}
\]
exists.  There exist an ergodic system $(X,\mu,T)$, a clopen set
$E\subseteq X$, a F{\o}lner sequence $\Psi$ in $\N$, and a point
$a\in\operatorname{gen}(\mu,\Psi)$ such that $\mu(E)\geq\delta$ and
\[
 A=N(a,E).
\]
\end{thm}

\begin{rem}
The proof of \cite[Theorem~2.10]{KMRR2024-1} is carried out in the full
binary shift, so the ambient space may be taken to be a Cantor space.
Replacing it by the two-sided orbit closure of $a$ makes $a$ transitive.
That orbit closure is compact and zero dimensional, although it need not be
a Cantor space if isolated points occur.
\end{rem}

\section{Correspondence between finite sumsets and adherence powers}\label{sec:dictionary}
In this section, we establish the correspondence between finite sumsets
and adherence powers and prove Theorem~\ref{main1}.

\subsection{Proof of Theorem \ref{main1}}
\begin{lem}\label{mainlem-1}\label{lem:cluster-transfer}
Let $(X,T)$ be a topological system, let $x\in X$, and let
$B\subseteq\N$ be infinite.  Suppose that $E_1,\ldots,E_k\subseteq X$
satisfy
\[
 B^{\oplus j}\subseteq N(x,E_j),\qquad 1\leq j\leq k.
\]
Then, for every $p\in\A_B(X,T)$,
\[
 (px,p^2x,\ldots,p^kx)
 \in\overline{E_1}\times\cdots\times\overline{E_k}.
\]
\end{lem}

\begin{proof}
Fix $j\in\{1,\ldots,k\}$.  By \eqref{eq:AB-net}, there is a net
$(c_\alpha)$ in $B$ such that $c_\alpha\to\infty$ and
$T^{c_\alpha}\to p$ pointwise. We prove by induction on
$r=0,1,\ldots,j$ that
\begin{equation}\label{eq:cluster-induction}
 T^{b_1+\cdots+b_{j-r}}p^rx\in\overline{E_j}
\end{equation}
whenever $b_1,\ldots,b_{j-r}$ are distinct elements of $B$. When $r=0$,
this is exactly the assumption $B^{\oplus j}\subseteq N(x,E_j)$.

Assume that \eqref{eq:cluster-induction} holds for some $r<j$, and fix
distinct $b_1,\ldots,b_{j-r-1}\in B$. Since $c_\alpha\to\infty$, the
value $c_\alpha$ is eventually different from all the $b_i$. Hence, by the
induction hypothesis,
\[
 T^{b_1+\cdots+b_{j-r-1}+c_\alpha}p^rx\in\overline{E_j}
\]
for all sufficiently large $\alpha$. Every element of $\A(X,T)$ commutes
with every power of $T$. Therefore
\[
 T^{b_1+\cdots+b_{j-r-1}+c_\alpha}p^rx
 \longrightarrow
 T^{b_1+\cdots+b_{j-r-1}}p^{r+1}x.
\]
Since $\overline{E_j}$ is closed, the induction step follows. Taking
$r=j$ gives $p^jx\in\overline{E_j}$. As $j$ was arbitrary, the proof is
complete.
\end{proof}

\begin{lem}\label{mainlem-2}\label{lem:multilevel-selection}
Let $(X,T)$ be a topological system, let $x\in X$, let
$p\in\A(X,T)$, and fix a net $(n_\alpha)$ such that
$n_\alpha\to\infty$ and $T^{n_\alpha}\to p$.  If $U_j$ is an open
neighborhood of $p^jx$ for $1\leq j\leq k$, then there is an infinite
set
\[
 B\subseteq\{n_\alpha:\alpha\in D\}
\]
such that
\[
 B^{\oplus j}\subseteq N(x,U_j),\qquad 1\leq j\leq k.
\]
\end{lem}

\begin{proof}
Put $x_i=p^ix$ for $0\leq i\leq k$.  We construct an increasing
sequence $b_1<b_2<\cdots$ from the range of $(n_\alpha)$.  At stage $m$
we require that, for every nonempty set $F\subseteq\{1,\ldots,m\}$ and
every $0\leq i\leq k-|F|$,
\begin{equation}\label{eq:multilevel-induction}
 T^{b_F}x_i\in U_{i+|F|},
 \qquad b_F:=\sum_{r\in F}b_r.
\end{equation}
There is no condition when $m=0$. Suppose that $b_1,\ldots,b_m$ have
been chosen. For every $F\subseteq\{1,\ldots,m\}$, including
$F=\varnothing$, and every $0\leq i\leq k-|F|-1$, we have
\[
 T^{n_\alpha+b_F}x_i
 \longrightarrow T^{b_F}px_i=T^{b_F}x_{i+1}.
\]
If $F=\varnothing$, the last point is $x_{i+1}\in U_{i+1}$. If
$F\neq\varnothing$, it belongs to $U_{i+|F|+1}$ by
\eqref{eq:multilevel-induction}, applied to $x_{i+1}$. There are only
finitely many pairs $(F,i)$. Since all the sets $U_j$ are open and
$n_\alpha\to\infty$, we may choose one value
$b_{m+1}=n_\alpha>b_m$ for which
\[
 T^{b_{m+1}+b_F}x_i\in U_{i+|F|+1}
\]
for all these pairs. This extends \eqref{eq:multilevel-induction} from
$m$ to $m+1$.

Let $B=\{b_1,b_2,\ldots\}$. Taking $i=0$ and $|F|=j$ in
\eqref{eq:multilevel-induction} gives
$B^{\oplus j}\subseteq N(x,U_j)$ for every $1\leq j\leq k$.
\end{proof}

\begin{proof}[Proof of Theorem~\ref{main1}]
Assume first that $p^jx\in E_j$ for $1\leq j\leq k$. Since the sets
$E_j$ are open, Lemma~\ref{mainlem-2}, with $U_j=E_j$, gives an infinite
set $B$ such that $B^{\oplus j}\subseteq N(x,E_j)$ for all $j$.
Conversely, assume that such a set $B$ exists and choose
$q\in\A_B(X,T)$. Lemma~\ref{mainlem-1} gives
$q^jx\in\overline{E_j}=E_j$, because the sets are clopen.
\end{proof}

Combining Lemmas~\ref{mainlem-1} and~\ref{mainlem-2} gives the following
general form of Theorem~\ref{main1}.
\begin{thm}\label{main1-general}\label{thm:dictionary}
Let $(X,T)$ be a topological system, let $x\in X$, and fix $k\in\N$.
For any nonempty subsets $E_1,\ldots,E_k\subseteq X$, we have the
following implications:
\begin{equation}\label{eq:general-target-dictionary}
\begin{gathered}
 \exists p\in\A(X,T):\quad
 p^jx\in\operatorname{int}(E_j)\quad(1\leq j\leq k)
 \\
 \Downarrow
 \\
 \exists B\subseteq\N\text{ infinite}:\quad
 B^{\oplus j}\subseteq N(x,E_j)\quad(1\leq j\leq k)
 \\
 \Downarrow
 \\
 \exists q\in\A(X,T):\quad
 q^jx\in\overline{E_j}\quad(1\leq j\leq k).
\end{gathered}
\end{equation}
\end{thm}

\begin{proof}
The first implication follows from Lemma~\ref{mainlem-2}, applied to the
open neighborhoods $\operatorname{int}(E_j)$ of $p^jx$. For the second
implication, choose any $q\in\A_B(X,T)$ and apply
Lemma~\ref{mainlem-1}. The closure in the last line is necessary for
general, nonclosed targets.
\end{proof}

In symbolic applications the targets $E_1,\ldots,E_k$ are usually
clopen.  In that case Theorem~\ref{thm:dictionary} reduces to the exact
equivalence
\[
\begin{gathered}
 \exists p\in\A(X,T):\quad p^jx\in E_j\quad(1\leq j\leq k)
 \\
 \Longleftrightarrow
 \\
 \exists B\subseteq\N\text{ infinite}:\quad
 B^{\oplus j}\subseteq N(x,E_j)\quad(1\leq j\leq k).
\end{gathered}
\]

For a common target $E\subseteq X$ and a shift $t\in\Z_+$,
Theorem~\ref{thm:dictionary} gives
\begin{equation}\label{eq:common-target-dictionary}
\begin{gathered}
 \exists p\in\A(X,T):\quad
 p^jx\in T^{-t}\operatorname{int}(E)\quad(1\leq j\leq k)
 \\
 \Downarrow
 \\
 \exists B\subseteq\N\text{ infinite}:\quad
 \displaystyle\bigcup_{j=1}^kB^{\oplus j}\subseteq N(x,E)-t
 \\
 \Downarrow
 \\
 \exists q\in\A(X,T):\quad
 q^jx\in T^{-t}\overline E\quad(1\leq j\leq k).
\end{gathered}
\end{equation}

The density finite sums theorem now yields the following return-set
consequence.
\begin{cor}\label{cor:positive-density-return-set}
Let $(X,T)$ be a topological dynamical system, let $x\in X$, and let
$E\subseteq X$.  If $N(x,E)$ has positive upper Banach density, then for
each $k\in\N$ there exist $t\in\Z_+$, an infinite set $B\subseteq\N$, and
$p\in\A(X,T)$ such that
\[
 \bigcup_{j=1}^{k}B^{\oplus j}\subseteq N(x,E)-t
 \quad\text{and}\quad
 p^jx\in T^{-t}\overline E\quad(1\leq j\leq k).
\]
\end{cor}

 \begin{proof}
 Apply Theorem~\ref{thm:density-fs} to $N(x,E)$. We obtain
 $t\in\Z_+$ and an infinite $B\subseteq\N$ such that
 \[
  \bigcup_{j=1}^kB^{\oplus j}\subseteq N(x,E)-t
  =N(x,T^{-t}E).
 \]
 Choose $p\in\A_B(X,T)$. Applying Lemma~\ref{mainlem-1} with
 $E_j=T^{-t}E$ for all $j$ gives
 $p^jx\in T^{-t}\overline E$ for $1\leq j\leq k$.
 \end{proof}
 
The following exact-level selection theorem is implicit in
Lemma~\ref{mainlem-2}; we include a direct proof for later use.
 \begin{thm}\label{thm:selection}
Let $(X,T)$ be a topological dynamical system, $p\in\A(X,T)$ and
$x_i=p^ix_0$ for $0\leq i\leq k$.  If $V$ is an open neighborhood of
$x_k$, then there is an infinite set $B\subseteq\N$ such that
\begin{equation}\label{eq:selection}
 B^{\oplus j}\subseteq N(x_{k-j},V),
 \qquad 1\leq j\leq k.
\end{equation}
Moreover, $B$ may be selected from the values of any net
$n_\alpha\to\infty$ satisfying $T^{n_\alpha}\to p$.
\end{thm}

\begin{proof}
Fix such a net, put $b_0=0$ and $F_0=\varnothing$, and construct
$b_1<b_2<\cdots$ inductively.  Suppose
$F_m=\{b_1,\ldots,b_m\}$ has been chosen and satisfies
\[
 F_m^{\oplus j}\subseteq N(x_{k-j},V)
 \quad(1\leq j\leq\min\{k,m\}).
\]
For $1\leq j\leq\min\{k,m+1\}$ and
$H\subseteq F_m$ with $|H|=j-1$, write $s_H=\sum_{h\in H}h$, with
$s_\varnothing=0$.  Since $p$ commutes with every power of $T$,
\[
 T^{n_\alpha+s_H}x_{k-j}\longrightarrow
 pT^{s_H}x_{k-j}=T^{s_H}x_{k-j+1}.
\]

We claim that $T^{s_H}x_{k-j+1}\in V$ for every such pair $(j,H)$.
For $j=1$, this is the assumption $x_k\in V$.  For $j\geq2$, it follows
from
\[
 F_m^{\oplus(j-1)}\subseteq N(x_{k-j+1},V).
\]

There are only finitely many pairs $(j,H)$ under consideration.  Since
$V$ is open and $n_\alpha\to\infty$,
we may choose a value $b_{m+1}>b_m$ for which all the corresponding points
$T^{b_{m+1}+s_H}x_{k-j}$ lie in $V$.  This extends the induction
hypothesis from $F_m$ to $F_{m+1}$.  Taking
$B=\{b_1,b_2,\ldots\}$ proves \eqref{eq:selection}.
\end{proof}

\subsection{Comparison with Erd\H{o}s progressions}
We use the definition of Erd\H{o}s progression from
\cite[Definition~2.1]{KMRR2026}. A tuple
$(x_0,x_1,\ldots,x_k)\in X^{k+1}$ is a $(k+1)$-term
\emph{Erd\H{o}s progression} if there is a strictly increasing sequence
$c\colon\N\to\N$ such that
\begin{equation}\label{eq:erdos-progression}
 T^{c(n)}x_{j-1}\longrightarrow x_j,
 \qquad 1\leq j\leq k,\quad n\to\infty.
\end{equation}
The next proposition identifies this sequential notion exactly with the
semigroup notion used here.

\begin{prop}\label{prop:erdos-adherence}
Let $(X,T)$ be a topological dynamical system and let
$(x_0,x_1,\ldots,x_k)\in X^{k+1}$.  The following are equivalent.
\begin{enumerate}
\item The tuple $(x_0,x_1,\ldots,x_k)$ is a $(k+1)$-term
Erd\H{o}s progression.
\item There is $p\in\A(X,T)$ such that
\[
 x_j=p^jx_0,\qquad 1\leq j\leq k.
\]
\end{enumerate}
If $c$ witnesses (1), then $p$ may be taken to be any pointwise cluster
point of the sequence $(T^{c(n)})$ in $X^X$.  Conversely, a sequence
witnessing (1) may be selected from the values of any net
$n_\alpha\to\infty$ for which $T^{n_\alpha}\to p$ pointwise.
\end{prop}

\begin{proof}
Assume first that \eqref{eq:erdos-progression} holds.  By the compactness of
$E(X,T)$, some subnet of $(T^{c(n)})$ converges pointwise to an element
$p\in E(X,T)$.  Since $c(n)\to\infty$, every such cluster point belongs to
$\A(X,T)$.  Restricting \eqref{eq:erdos-progression} to this subnet gives
\[
 px_{j-1}=x_j,\qquad 1\leq j\leq k.
\]
Induction yields $x_j=p^jx_0$.

Conversely, suppose $x_j=p^jx_0$ and fix a net
$n_\alpha\to\infty$ such that $T^{n_\alpha}\to p$.  Let $d$ be a
compatible metric on $X$.  Put $c(0)=0$ as an auxiliary initial value
and, for $r\geq1$, recursively choose values $c(r)=n_{\alpha_r}$
so that $c(r)>c(r-1)$ and
\[
 d\bigl(T^{c(r)}p^{j-1}x_0,p^jx_0\bigr)<\frac1r,
 \qquad 1\leq j\leq k.
\]
This is possible because only finitely many evaluation points are involved
and $n_\alpha\to\infty$.  The resulting strictly increasing sequence
satisfies \eqref{eq:erdos-progression}.
\end{proof}

There are two useful consequences of Proposition
\ref{prop:erdos-adherence}.
\begin{enumerate}
\item Under this identification, \cite[Lemma~2.2]{KMRR2026} is precisely
the open-target implication in Lemma \ref{lem:multilevel-selection}: from
$p^jx_0\in U_j$ for $1\leq j\leq k$, one obtains a single infinite set
$B$ such that
\[
 B^{\oplus j}\subseteq N(x_0,U_j),\qquad 1\leq j\leq k.
\]
Lemma \ref{lem:cluster-transfer} supplies the reverse passage from finite
sumsets to an Erd\H{o}s progression, equivalently to an adherence-power
progression, with closure of the targets.  Theorem \ref{thm:dictionary}
records both directions and isolates the boundary issue.

\item The dynamical conclusion of \cite[Theorems~1.2 and~2.4]{KMRR2026}
can therefore be written in the semigroup form
\begin{equation}\label{eq:kmrr-semigroup-form}
 \exists t\in\Z_+\ \text{and}\  p\in\A(X,T):\quad
 p^ja\in T^{-t}E,\qquad 1\leq j\leq k.
\end{equation}
Indeed, their Erd\H{o}s progression $(a,x_1,\ldots,x_k)$ has
$x_j=p^ja$.  Conversely, any $p$ in \eqref{eq:kmrr-semigroup-form}
produces such an Erd\H{o}s progression.  Thus Corollary
\ref{cor:finite-powers} is the adherence-semigroup formulation of the same
finite dynamical configuration.  The logical roles are different:
\cite{KMRR2026} proves its Erd\H{o}s-progression theorem as the main
dynamical input to the density finite sums theorem, whereas the present note
takes the density finite sums theorem as an input and derives
\eqref{eq:kmrr-semigroup-form} from it.
\end{enumerate}

\subsection{Combinatorial and dynamical equivalence}

\begin{prop}
\label{thm:equivalence}
For each fixed $k\in\N$, the following statements are equivalent.
\begin{enumerate}
\item[${\rm (C)}_k$] Every $A\subseteq\N$ with $d^*(A)>0$ admits an
infinite $B\subseteq\N$ and $t\in\Z_{+}$ such that
$B^{\oplus k}\subseteq A-t$.
\item[${\rm (D)}_k$] Let $(X,\mu,T)$ be a system, let
$a\in\operatorname{gen}(\mu,\Phi)$ for some F{\o}lner sequence $\Phi$, and let
$E\subseteq X$ be open with $\mu(E)>0$. Then there exist
$p\in\A(X,T)$ and $t\in\Z_{+}$ such that
$p^ka\in T^{-t}E$.
\end{enumerate}
\end{prop}

\begin{proof}
Assume ${\rm (C)}_k$.  Choose
$y\in E\cap\operatorname{supp}(\mu)$ and an open neighborhood $U$ of $y$
such that $\overline U\subseteq E$.  Then $\mu(U)>0$, and
Lemma~\ref{lem:generic-return-density} shows that $N(a,U)$ has positive
upper Banach density.
Apply ${\rm (C)}_k$ to obtain $B$ and $t$ with
$B^{\oplus k}\subseteq N(a,U)-t=N(a,T^{-t}U)$.  For any
$p\in\A_B(X,T)$, Lemma \ref{lem:cluster-transfer}, applied with
$E_k=T^{-t}U$ and $E_j=X$ for $1\leq j<k$, gives
\[
 p^ka\in T^{-t}\overline U\subseteq T^{-t}E.
\]

Conversely, suppose ${\rm (D)}_k$ holds and let $A\subseteq\N$ have
positive upper Banach density. By \eqref{eq:banach-folner}, choose a
F{\o}lner sequence $\Phi$ along which the upper density of $A$ is
positive. Passing to a subsequence, which remains a F{\o}lner sequence,
we may assume that
\[
 \delta=\lim_{N\to\infty}\frac{|A\cap\Phi_N|}{|\Phi_N|}>0.
\]
Apply Theorem~\ref{thm:correspondence}.
We obtain an ergodic system $(X,\mu,T)$, a clopen $E$, and a generic
point $a$ such that $\mu(E)>0$ and $A=N(a,E)$.  By ${\rm (D)}_k$, there
are $p\in\A(X,T)$ and $t\in\Z_{+}$ with $p^ka\in T^{-t}E$.
Apply Theorem \ref{thm:selection} with $x_0=a$ and $V=T^{-t}E$.  It
produces an infinite $B$ satisfying
\[
 B^{\oplus k}\subseteq N(a,T^{-t}E)=A-t.
\]
\end{proof}
We next replace the positive-density hypothesis by a finite partition of the
natural numbers.  The exact-level analogue of Theorem
\ref{thm:equivalence} has a particularly simple dynamical side.

\begin{thm}
\label{thm:partition-exact}
Fix $k\in\N$.  The following statements are equivalent.
\begin{enumerate}
\item[${\rm (PR)}_k$] For every finite partition
\[
 \N=C_1\sqcup\cdots\sqcup C_r,
\]
there exist $i\in\{1,\ldots,r\}$ and an infinite set $B\subseteq\N$ such
that $B^{\oplus k}\subseteq C_i$.
\item[${\rm (DPR)}_k$] For every topological dynamical system $(X,T)$,
every $x\in X$, and every finite clopen partition
\[
 X=E_1\sqcup\cdots\sqcup E_r,
\]
there exist $i\in\{1,\ldots,r\}$ and $p\in\A(X,T)$ such that
$p^kx\in E_i$.
\end{enumerate}
Moreover, these equivalent statements hold for every $k\in\N$.
\end{thm}

\begin{proof}
Assume ${\rm (PR)}_k$.  Given $(X,T)$, $x$, and a clopen partition as in
${\rm (DPR)}_k$, the sets
\[
 C_i=N(x,E_i),\qquad 1\leq i\leq r,
\]
form a partition of $\N$.  Choose $i$ and an infinite $B$ with
$B^{\oplus k}\subseteq C_i$.  Choose $p\in\A_B(X,T)$ and apply
Lemma~\ref{lem:cluster-transfer} with $E_k=E_i$ and $E_j=X$ for
$1\leq j<k$.  Closedness of $E_i$ gives $p^kx\in E_i$.

Conversely, let $\N=C_1\sqcup\cdots\sqcup C_r$.  Extend its coloring
arbitrarily to the nonpositive integers and encode it by a point
$a\in\{1,\ldots,r\}^{\Z}$.  Let $T$ be the shift, let $X$ be the orbit
closure of $a$, and put
\[
 E_i=\{z\in X:z_0=i\}.
\]
Then $E_1,\ldots,E_r$ form a clopen partition of $X$ and
$N(a,E_i)=C_i$.  By ${\rm (DPR)}_k$, some $p\in\A(X,T)$ satisfies
$p^ka\in E_i$.  Theorem \ref{thm:selection}, applied with $x_0=a$ and
$V=E_i$, produces an infinite $B$ such that
$B^{\oplus k}\subseteq N(a,E_i)=C_i$.

Finally, $\A(X,T)$ is nonempty.  For any $p\in\A(X,T)$, the point
$p^kx$ lies in one member of the clopen partition.  Thus
${\rm (DPR)}_k$ always holds, completing the proof.
\end{proof}

The multilevel form of Theorem \ref{thm:dictionary}, together with an
idempotent adherence element, yields the stronger simultaneous statement.

\begin{thm}
\label{thm:partition-multilevel}
Fix $k\in\N$.  The following statements are equivalent.
\begin{enumerate}
\item[${\rm (PR)}_{\leq k}$] For every finite partition
$\N=C_1\sqcup\cdots\sqcup C_r$, there exist an index $i$ and an infinite
$B\subseteq\N$ such that
\[
 \bigcup_{j=1}^kB^{\oplus j}\subseteq C_i.
\]
\item[${\rm (DPR)}_{\leq k}$] For every topological dynamical system
$(X,T)$, every $x\in X$, and every finite clopen partition
$X=E_1\sqcup\cdots\sqcup E_r$, there exist an index $i$ and
$p\in\A(X,T)$ such that
\[
 p^jx\in E_i,\qquad 1\leq j\leq k.
\]
\end{enumerate}
Moreover, these equivalent statements hold for every $k\in\N$.
\end{thm}

\begin{proof}
The implication ${\rm (PR)}_{\leq k}\Rightarrow
{\rm (DPR)}_{\leq k}$ follows by applying
${\rm (PR)}_{\leq k}$ to the partition $C_i=N(x,E_i)$, choosing
$p\in\A_B(X,T)$, and using Lemma \ref{lem:cluster-transfer} at every
level $1\leq j\leq k$.

For the converse, encode a partition of $\N$ by the two-sided symbolic
system used in the proof of Theorem \ref{thm:partition-exact}.  If
${\rm (DPR)}_{\leq k}$ gives $p^ja\in E_i$ for $1\leq j\leq k$,
then Lemma \ref{lem:multilevel-selection}, with $U_j=E_i$, produces an
infinite $B$ satisfying
\[
 \bigcup_{j=1}^kB^{\oplus j}\subseteq N(a,E_i)=C_i.
\]

It remains to verify the dynamical assertion.  By the Ellis--Numakura lemma,
the compact right-topological semigroup $\A(X,T)$ contains an idempotent
$u$; see, for example, \cite[Corollary~2.6]{HS2012}. Choose $i$ such
that $ux\in E_i$.
Since $u^j=u$ for every $j\geq1$, one has
$u^jx=ux\in E_i$ for $1\leq j\leq k$.  Thus
${\rm (DPR)}_{\leq k}$ holds.
\end{proof}

\begin{rem}
The exact-level assertion follows from $\A(X,T)\neq\varnothing$. By contrast, the multilevel assertion uses the semigroup
structure in an essential way: its proof requires an idempotent. Requiring
one infinite set $B$ to work simultaneously for every $k$ is Hindman's
finite-sums theorem \cite{Hindman1974}. Its proof requires the usual
idempotent recursion rather than applying Theorem
\ref{thm:partition-multilevel} separately for each $k$.
\end{rem}

Theorem \ref{thm:density-fs} and the stronger, multi-level part of
Lemma \ref{lem:cluster-transfer} give the following simultaneous version.

\begin{cor}\label{cor:finite-powers}
Let $(X,\mu,T)$ be a system, let
$a\in\operatorname{gen}(\mu,\Phi)$, and let
$E\subseteq X$ be open with $\mu(E)>0$. For every $k\in\N$ there exist
$p\in\A(X,T)$ and $t\in\Z_+$ such that
\[
 p^ja\in T^{-t}E,\qquad 1\leq j\leq k.
\]
\end{cor}

\begin{proof}
Choose an open $U$ with $\mu(U)>0$ and $\overline U\subseteq E$ as in
the proof of Theorem \ref{thm:equivalence}.  By
Lemma~\ref{lem:generic-return-density}, $N(a,U)$ has positive upper
Banach density.  Theorem \ref{thm:density-fs} gives an infinite
$B\subseteq\N$ and $t\in\Z_+$ such that
\[
 \bigcup_{j=1}^kB^{\oplus j}
 \subseteq N(a,U)-t=N(a,T^{-t}U).
\]
Choose $p\in\A_B(X,T)$.  Lemma \ref{lem:cluster-transfer}, applied with
$E_j=T^{-t}U$ for $1\leq j\leq k$, gives
$p^ja\in T^{-t}\overline U\subseteq T^{-t}E$.
\end{proof}

\begin{rem}
Ergodicity is not needed in either ${\rm (D)}_k$ or Corollary
\ref{cor:finite-powers}; the only measure-theoretic input in these directions
is the existence of the stated generic point.
\end{rem}

\section{Applications}

In this section, we give applications of the tools developed in the preceding
sections.

\subsection{An ultrafilter consequence}

Let $\beta\N$ denote the Stone--\v{C}ech compactification of $\N$, identified
with the space of ultrafilters on $\N$, and put
$\N^*:=\beta\N\setminus\N$.  For the algebra of $\beta\N$ and its
Ramsey-theoretic applications, see \cite{HS2012}; see also
\cite[Section~5.5]{KMRRProblems2025}.  Addition on $\beta\N$ is defined by
\[
 C\in\mathfrak r+\mathfrak s
 \quad\Longleftrightarrow\quad
 \{n\in\N:C-n\in\mathfrak s\}\in\mathfrak r,
\]
for any ultrafilter $\mathfrak r, \mathfrak s\in \beta\mathbb{N}$.

For $t\geq1$, we identify $t$ with the principal ultrafilter at $t$;
addition by $0$ is interpreted as the identity.  It follows directly from
the definition that, for every $\mathfrak r\in\beta\N$ and $t\in\Z_+$,
\begin{equation}\label{eq:principal-translation}
 C\in\mathfrak r+t
 \quad\Longleftrightarrow\quad
 C-t\in\mathfrak r
 \quad\Longleftrightarrow\quad
 C\in t+\mathfrak r.
\end{equation}
Thus every principal ultrafilter lies in the center of $\beta\N$.

For a compact Hausdorff space $X$ and a sequence $(x_n)_{n\in\N}$ in
$X$, its ultrafilter limit along $\mathfrak r\in\beta\N$, denoted by
$\mathfrak r\text{-}\lim_nx_n$, is the unique point $x$ such that
\[
 \{n\in\N:x_n\in U\}\in\mathfrak r
\]
for every neighborhood $U$ of $x$.

We shall use the standard bridge from the adherence semigroup of
Section~\ref{sec:dictionary} to $\beta\N$.  For a topological dynamical
system $(X,T)$, define
\[
 \rho\colon\beta\N\longrightarrow E(X,T),\qquad
 \rho(\mathfrak r)x=\mathfrak r\text{-}\lim_nT^nx.
\]

\begin{lem}\label{lem:beta-ellis}
The map $\rho$ is a semigroup homomorphism, and its restriction to
$\N^*$ maps onto the adherence semigroup.  More precisely,
\begin{equation}\label{eq:beta-to-ellis}
 \rho(\mathfrak r+\mathfrak s)
 =\rho(\mathfrak r)\rho(\mathfrak s),
 \qquad
 \rho(\N^*)=\A(X,T).
\end{equation}
\end{lem}

\begin{proof}
For $x\in X$, the definition of addition gives
\[
 \rho(\mathfrak r+\mathfrak s)x
 =\mathfrak r\text{-}\lim_n\,
   \mathfrak s\text{-}\lim_m T^{n+m}x
 =\mathfrak r\text{-}\lim_n T^n\rho(\mathfrak s)x
 =\rho(\mathfrak r)\rho(\mathfrak s)x.
\]
If
$\mathfrak u\in\N^*$, every tail $\{n\in\N:n\geq m\}$ belongs to
$\mathfrak u$, so $\rho(\mathfrak u)$ belongs to every closed set in
\eqref{eq:adherence}.

Conversely, let $p\in\A(X,T)$.  For each neighborhood $\mathcal U$ of
$p$ in $X^X$ and each $m\in\N$, set
\[
 H(\mathcal U,m):=\{n\geq m:T^n\in\mathcal U\}.
\]
These sets have the finite-intersection property by
\eqref{eq:adherence}.  Extend them to an ultrafilter $\mathfrak u$.
Since $\mathfrak u$ contains every tail, it is nonprincipal; and since it
contains $H(\mathcal U,1)$ for every $\mathcal U$, one has
$\rho(\mathfrak u)=p$.
\end{proof}

\begin{proof}[Proof of Theorem \ref{main2}]
We derive the result directly from the positive-density return-set form of
the dictionary, Corollary~\ref{cor:positive-density-return-set}.  Let
\[
 X=\{0,1\}^{\Z},\qquad (Tx)_m=x_{m+1},
\]
let $a\in X$ be defined by $a_n=1_A(n)$ for $n\in\N$ (with arbitrary
coordinates for $n\leq0$), and put
\[
 E=\{x\in X:x_0=1\}.
\]
Then $E$ is clopen and $N(a,E)=A$.  Apply
Corollary~\ref{cor:positive-density-return-set} with $k=2$.  Since
$d^*(A)>0$ and $\overline E=E$, it gives $t\in\Z_+$ and
$p\in\A(X,T)$ such that
\begin{equation}\label{eq:section-three-power}
 p^ja\in T^{-t}E,\qquad j=1,2.
\end{equation}

By Lemma~\ref{lem:beta-ellis}, choose a nonprincipal
$\mathfrak u\in\beta\N$ with $\rho(\mathfrak u)=p$.  Again by
\eqref{eq:beta-to-ellis},
\[
 \rho(\mathfrak u+\mathfrak u)=p^2.
\]
For every clopen set $U\subseteq X$ and every
$\mathfrak r\in\beta\N$, the definition of the ultrafilter limit gives
\[
 \rho(\mathfrak r)a\in U
 \quad\Longleftrightarrow\quad
 N(a,U)\in\mathfrak r.
\]
Taking $U=T^{-t}E$ and using the case $j=2$ of
\eqref{eq:section-three-power}, we therefore obtain
\begin{equation}\label{eq:D-in-uu}
 A-t=N(a,T^{-t}E)\in\mathfrak u+\mathfrak u.
\end{equation}

Set
\[
 \mathfrak p=\mathfrak u,
 \qquad
 \mathfrak q=\mathfrak u+t.
\]
The ultrafilter $\mathfrak p$ is nonprincipal by construction, and
$\mathfrak q$ is nonprincipal because translation by the principal
ultrafilter $t$ preserves nonprincipality.  Indeed, if a singleton belonged
to $\mathfrak u+t$, then \eqref{eq:principal-translation} would put a
finite set in $\mathfrak u$.

We next check commutativity.  By the associativity and the centrality of the
principal ultrafilter $t$,
\[
 \mathfrak p+\mathfrak q
 =\mathfrak u+(\mathfrak u+t)
 =(\mathfrak u+\mathfrak u)+t,
\]
whereas
\[
 \mathfrak q+\mathfrak p
 =(\mathfrak u+t)+\mathfrak u
 =\mathfrak u+(t+\mathfrak u)
 =\mathfrak u+(\mathfrak u+t)
 =(\mathfrak u+\mathfrak u)+t.
\]
Thus $\mathfrak p+\mathfrak q=\mathfrak q+\mathfrak p$.  Finally,
applying \eqref{eq:principal-translation} with
$\mathfrak r=\mathfrak u+\mathfrak u$ and using
\eqref{eq:D-in-uu}, we obtain
\[
 A\in(\mathfrak u+\mathfrak u)+t
 \quad\Longleftrightarrow\quad
 A-t\in\mathfrak u+\mathfrak u.
\]
Since $\mathfrak p+\mathfrak q=(\mathfrak u+\mathfrak u)+t$, this proves
$A\in\mathfrak p+\mathfrak q$ and completes the proof.
\end{proof}

\subsection{A dynamical criterion through adherence powers}
\label{subsec:dynamical-ultrafilter-criterion}

To prove Theorem~\ref{main3}, we will use the following result.

\begin{lem}[
{\cite[Corollary~6.21]{HS2012}}]
\label{lem:left-ideal-comparability}
Let $\mathfrak p,\mathfrak q\in\beta\N$ be distinct.  If
\[
 (\beta\N+\mathfrak p)\cap(\beta\N+\mathfrak q)\neq\varnothing,
\]
then
\[
 \mathfrak p\in\beta\N+\mathfrak q
 \quad\text{or}\quad
 \mathfrak q\in\beta\N+\mathfrak p.
\]
\end{lem}

Let $A\subseteq\N$.  Consider the two-sided symbolic system
\[
 \Sigma=\{0,1\}^{\Z},\qquad (Tz)_j=z_{j+1},
\]
and let $a\in\Sigma$ be the indicator of $A$, where $a_j=0$ for
$j\leq0$.  Put
\[
 X_A=\overline{\{T^na:n\in\Z\}},
 \qquad
 E_A=\{z\in X_A:z_0=1\}.
\]
Then $E_A$ is clopen and $N(a,E_A)=A$.

The following proposition supplies both the dynamical formulation and the
proof of Theorem~\ref{main3}.
\begin{prop}
\label{prop:dynamical-commuting-pair}
For $A\subseteq\N$, the following statements are equivalent:
\begin{enumerate}
\item there exist an infinite set $B\subseteq\N$ and $t\in\Z_+$ such that
      \[
       B^{\oplus2}\subseteq A-t;
      \]
\item there exist $e\in\A(X_A,T)$ and $t\in\Z_+$ such that
      \[
       e^2a\in T^{-t}E_A;
      \]
\item there exist a nonprincipal ultrafilter
      $\mathfrak u\in\beta\N\setminus\N$ and $t\in\Z_+$ such that
      \[
       A-t\in\mathfrak u+\mathfrak u;
      \]
\item there exist $\mathfrak u\in\beta\N\setminus\N$ and $t\in\Z_+$ such
      that, on setting
      \[
       \mathfrak p=\mathfrak u,
       \qquad
       \mathfrak q=\mathfrak u+t,
      \]
      the ultrafilters $\mathfrak p$ and $\mathfrak q$ are nonprincipal
      and satisfy
      \[
       A\in \mathfrak p+\mathfrak q=\mathfrak q+\mathfrak p.
      \]
\item there exist nonprincipal ultrafilters
      $\mathfrak p,\mathfrak q\in\beta\N$ such that
      \[
       A\in \mathfrak p+\mathfrak q=\mathfrak q+\mathfrak p.
      \]
\end{enumerate}
\end{prop}

\begin{proof}
Fix $t\in\Z_+$.  Since
\[
 N(a,T^{-t}E_A)=A-t,
\]
Theorem~\ref{main1}, applied with $k=2$, $E_1=X_A$ and
$E_2=T^{-t}E_A$, gives
\[
 \begin{split}
 &\text{there is an infinite }B\subseteq\N
   \text{ with }B^{\oplus2}\subseteq A-t\\
 &\hspace{35mm}\Longleftrightarrow
   \text{there is }e\in\A(X_A,T)
   \text{ with }e^2a\in T^{-t}E_A.
 \end{split}
\]
This proves the equivalence of (1) and (2).

For a clopen set $V\subseteq X_A$ and
$\mathfrak r\in\beta\N$, the definition of $\rho$ gives
\begin{equation}\label{eq:clopen-ultrafilter-criterion}
 \rho(\mathfrak r)a\in V
 \quad\Longleftrightarrow\quad
 N(a,V)\in\mathfrak r.
\end{equation}
By \eqref{eq:beta-to-ellis}, every $e\in\A(X_A,T)$ has the form
$e=\rho(\mathfrak u)$ for some nonprincipal
$\mathfrak u\in\beta\N\setminus\N$, and
\[
 \rho(\mathfrak u+\mathfrak u)=e^2.
\]
Taking $V=T^{-t}E_A$ in
\eqref{eq:clopen-ultrafilter-criterion}, we obtain
\[
 e^2a\in T^{-t}E_A
 \quad\Longleftrightarrow\quad
 A-t\in\mathfrak u+\mathfrak u.
\]
Hence (2) and (3) are equivalent.

Assume (3), and define
$\mathfrak p=\mathfrak u$ and
$\mathfrak q=\mathfrak u+t$.  Translation by a principal ultrafilter
preserves nonprincipality, so both $\mathfrak p$ and $\mathfrak q$ are
nonprincipal.  Associativity and the centrality of the principal
ultrafilter $t$ give
\[
 \mathfrak p+\mathfrak q
 =\mathfrak u+(\mathfrak u+t)
 =(\mathfrak u+\mathfrak u)+t
\]
and
\[
 \mathfrak q+\mathfrak p
 =(\mathfrak u+t)+\mathfrak u
 =\mathfrak u+(t+\mathfrak u)
 =\mathfrak u+(\mathfrak u+t)
 =(\mathfrak u+\mathfrak u)+t.
\]
Thus $\mathfrak p+\mathfrak q=\mathfrak q+\mathfrak p$.  Moreover,
\eqref{eq:principal-translation} yields
\[
 A-t\in\mathfrak u+\mathfrak u
 \quad\Longleftrightarrow\quad
 A\in(\mathfrak u+\mathfrak u)+t
 =\mathfrak p+\mathfrak q.
\]
This proves (4).  Conversely, if (4) holds for a pair of the displayed
special form, the same identities, read backwards, give
$A-t\in\mathfrak u+\mathfrak u$.  Therefore (3) and (4) are equivalent.
Clearly, (4) implies (5).

It remains to prove that (5) implies (3).  Let
\[
 \mathfrak r
 :=\mathfrak p+\mathfrak q
 =\mathfrak q+\mathfrak p
\]
and suppose that $A\in\mathfrak r$.  If
$\mathfrak p=\mathfrak q$, then (3) holds with
$\mathfrak u=\mathfrak p$ and $t=0$.

Assume now that $\mathfrak p\neq\mathfrak q$.  Since
\[
 \mathfrak r\in\beta\N+\mathfrak q
 \quad\text{and}\quad
 \mathfrak r\in\beta\N+\mathfrak p,
\]
Lemma~\ref{lem:left-ideal-comparability} gives either
\[
 \mathfrak q=\mathfrak s+\mathfrak p
 \quad\text{or}\quad
 \mathfrak p=\mathfrak s+\mathfrak q
\]
for some $\mathfrak s\in\beta\N$.  In the first case, associativity gives
\[
 \mathfrak r
 =\mathfrak q+\mathfrak p
 =\mathfrak s+(\mathfrak p+\mathfrak p).
\]
Consequently,
\[
 H:=\{t\in\N:A-t\in\mathfrak p+\mathfrak p\}
 \in\mathfrak s.
\]
Choose $t\in H$ and put $\mathfrak u=\mathfrak p$.  Then
$A-t\in\mathfrak u+\mathfrak u$, which is (3).  In the second case,
the same argument, using
\[
 \mathfrak r
 =\mathfrak p+\mathfrak q
 =\mathfrak s+(\mathfrak q+\mathfrak q),
\]
gives (3) with $\mathfrak u=\mathfrak q$.  This proves
${\rm (5)}\Rightarrow{\rm (3)}$ and completes the equivalence.

Conditions (1) and (5) are precisely the two conditions in
Theorem~\ref{main3}; hence the proposition also proves that theorem.
\end{proof}

\begin{proof}[A second proof of Theorem~\ref{main2}]
Apply Theorem~\ref{thm:density-fs} with $k=2$. There exist an infinite
$B\subseteq\N$ and $t\in\Z_+$ such that
$B^{\oplus2}\subseteq A-t$.  The implication
${\rm (1)}\Rightarrow{\rm (4)}$ in
Proposition~\ref{prop:dynamical-commuting-pair} produces nonprincipal
ultrafilters
\[
 \mathfrak p=\mathfrak u,
 \qquad
 \mathfrak q=\mathfrak u+t
\]
such that
\[
 \mathfrak p+\mathfrak q=\mathfrak q+\mathfrak p
 \quad\text{and}\quad
 A\in\mathfrak p+\mathfrak q.
\]
\end{proof}

\subsection{Power density in minimal systems}
Let $(X,T)$ be a topological dynamical system. For $x\in X$ and
$k\in\N$, denote
\[
 P_k(x):=\{p^kx:p\in\A(X,T)\}\ \text{ and }\ X_k(x)=\overline{\{T^{kn}x:n\in\N\}}.
\]

\begin{thm}
\label{thm:minimal-power-density}
Let $(X,T)$ be a minimal system, let $x\in X$, and let $k\in\N$. Then
\[
 \overline{P_k(x)}=X_k(x).
\]
In particular, the following are equivalent:
\begin{enumerate}
\item $(X,T^k)$ is minimal;
\item $\overline{P_k(x)}=X$ for every $x\in X$;
\item $\overline{P_k(x)}=X$ for some $x\in X$.
\end{enumerate}
Consequently, a minimal system $(X,T)$ is totally minimal if and only if
\[
 \overline{\{p^k x:p\in\A(X,T)\}}=X
\]
for every $k\in\N$ and every $x\in X$. Equivalently, it is enough that
this condition hold for every $k\in\N$ at one fixed point $x\in X$.
\end{thm}

\begin{proof}
Choose a minimal subset $Y$ for the action of $T^k$. Since $T$ commutes
with $T^k$, each $T^jY$ is also $T^k$-minimal, and minimality of $T$
gives
\[
 X=\bigcup_{j=0}^{k-1}T^jY.
\]
Let $d$ be the least positive integer such that $T^dY=Y$. Then $d\mid k$,
and the sets $Y,TY,\ldots,T^{d-1}Y$ are pairwise disjoint clopen sets.
Indeed, two intersecting members of this family are equal because both are
$T^k$-minimal; minimality of $d$ then gives disjointness, and the finite
closed partition makes every member clopen.  They form the cyclic
decomposition associated with the power $T^k$. If $x\in T^jY$, then
\[
 X_k(x)=T^jY.
\]

Let $\pi:X\to\Z/d\Z$ be the factor map determined by this decomposition,
so that $\pi(Tx)=\pi(x)+1$. Take $p\in\A(X,T)$ and choose a net $(n_\alpha)$ with $n_{\alpha}\to\infty$ such that 
$T^{n_\alpha}\to p$. Passing to a subnet, we may suppose that
$n_\alpha\equiv r\pmod d$. Hence
\[
 \pi(px)=\pi(x)+r,
 \qquad
 \pi(p^kx)=\pi(x)+kr=\pi(x),
\]
because the first identity holds with $x$ replaced by any point of $X$ and
$d\mid k$. Therefore $p^kx\in X_k(x)$, and
\[
 \overline{P_k(x)}\subseteq X_k(x).
\]

For the reverse inclusion, let $U$ be a nonempty relatively open subset of
$X_k(x)$ and choose a nonempty relatively open set $V$ with
$\overline V^{\,X_k(x)}\subseteq U$. The system $(X_k(x),T^k)$ is
minimal.  Since $X_k(x)$ is closed in $X$, the relative closure of $V$
is also its closure in $X$.  Moreover,
\[
 C=N_{T^k}(x,V)=\{m\in\N:T^{km}x\in V\}
\]
is syndetic and hence has positive upper Banach density. Apply
Theorem~\ref{thm:density-fs} to $C$ and retain its exact $k$-fold
conclusion. Thus there exist an infinite $D\subseteq\N$ and $s\in\Z_+$
such that $D^{\oplus k}\subseteq C-s$. Put
\[
 B=kD+s:=\{kd+s:d\in D\}.
\]
For distinct $d_1,\ldots,d_k\in D$,
\[
 (kd_1+s)+\cdots+(kd_k+s)
 =k(d_1+\cdots+d_k+s),
\]
and $d_1+\cdots+d_k+s\in C$. Hence
$B^{\oplus k}\subseteq N(x,V)$. Choose $p\in\A_B(X,T)$ and apply
Lemma~\ref{mainlem-1}, taking $E_k=V$ and $E_j=X$ for $j<k$. We obtain
$p^kx\in\overline V^{\,X_k(x)}\subseteq U$. Since $U$ was arbitrary,
$X_k(x)\subseteq\overline{P_k(x)}$.

If $(X,T^k)$ is minimal, then $X_k(x)=X$ for every $x$. Conversely, if
$X_k(x)=X$ for even one $x$, the cyclic component containing $x$ is all
of $X$, so $Y=X$ and $(X,T^k)$ is minimal.  Together with
$\overline{P_k(x)}=X_k(x)$, this proves the three equivalences. Applying
them for every $k$ gives the characterization of total minimality.
\end{proof}

\subsection{One adherence element for all orders}

Hern\'andez, Kousek and Radi\'c proved that the infinite set generating the
exact $k$-fold sumsets can be chosen independently of $k$, provided the shift
is allowed to depend on $k$ \cite[Theorem A, with $\ell=0$]{HKR2025}.
The cluster lemma gives the following dynamical formulation.

\begin{thm}\label{thm:one-p-all-orders}
Let $(X,\mu,T)$ be a system, let
$a\in\operatorname{gen}(\mu,\Phi)$, and let
$E\subseteq X$ be open with $\mu(E)>0$. There exist
$p\in\A(X,T)$ and a sequence $(t_k)_{k\geq1}$ in $\Z_+$ such that
\[
 p^ka\in T^{-t_k}E\qquad\text{for every }k\geq1.
\]
\end{thm}

\begin{proof}
Choose an open $U$ with $\mu(U)>0$ and $\overline U\subseteq E$.  The set
$A=N(a,U)$ has positive upper Banach density by
Lemma~\ref{lem:generic-return-density}.  By
\cite[Theorem~A, with $\ell=0$]{HKR2025}, there exist an infinite
$B\subseteq\N$ and shifts
$t_k\in\Z_+$ such that
\[
 B^{\oplus k}\subseteq A-t_k=N(a,T^{-t_k}U)
 \qquad(k\geq1).
\]
Choose $p\in\A_B(X,T)$.  For each fixed $k$, apply
Lemma~\ref{lem:cluster-transfer} with
$E_k=T^{-t_k}U$ and $E_j=X$ for $1\leq j<k$.  This yields
$p^ka\in T^{-t_k}\overline U\subseteq T^{-t_k}E$.
\end{proof}

\begin{rem}
The full form of \cite[Theorem~A]{HKR2025} implies that for every fixed
$\ell\geq0$ there exist $p_\ell\in\A(X,T)$ and shifts
$t_{\ell,k}\in\Z_+$ such
that
\[
 p_\ell^ia\in T^{-t_{\ell,k}}E
 \qquad(k\leq i\leq k+\ell)
\]
for every $k\geq1$.
\end{rem}

The same paper proves the existence of infinite sets $B_1,B_2,\ldots$ such
that $B_1+\cdots+B_m$ is contained in a given positive-density set for every
$m$ \cite[Corollary~C]{HKR2025}. This yields a phenomenon that records the
possible noncommutativity of the adherence semigroup.

\begin{cor}\label{cor:product-tower}
Under the hypotheses of Theorem~\ref{thm:one-p-all-orders}, there exists a
sequence $(p_i)_{i\geq1}$ in $\A(X,T)$ such that, for every $m\geq1$ and
every permutation $\sigma$ of $\{1,\ldots,m\}$,
\[
 p_{\sigma(1)}p_{\sigma(2)}\cdots p_{\sigma(m)}a\in E.
\]
\end{cor}

\begin{proof}
Choose an open set $U$ with $\mu(U)>0$ and
$\overline U\subseteq E$.  By Lemma~\ref{lem:generic-return-density},
$A=N(a,U)$ has positive upper Banach density.  Choose infinite sets
$B_1,B_2,\ldots$ such that $B_1+\cdots+B_m\subseteq A$ for every
$m$, as provided by \cite[Corollary~C]{HKR2025}. Choose
$p_i\in\A_{B_i}(X,T)$.

Fix $m$ and a permutation $\sigma$ of $\{1,\ldots,m\}$. For every
choice $b_i\in B_i$ we have
$T^{b_1+\cdots+b_m}a\in U$. By \eqref{eq:AB-net}, choose for every $i$
a net in $B_i$ tending to infinity whose powers converge to $p_i$.  First take the
limit in the variable $b_{\sigma(m)}$, keeping all other variables fixed.
Since $\overline U$ is closed, this gives
\[
 T^{\sum_{i\neq\sigma(m)}b_i}p_{\sigma(m)}a\in\overline U.
\]
Next take the cluster limits successively in the variables
\[
 b_{\sigma(m-1)},\ldots,b_{\sigma(1)}.
\]
Powers of $T$ commute with each adherence element, so the final limit is
\[
 p_{\sigma(1)}p_{\sigma(2)}\cdots p_{\sigma(m)}a
 \in\overline U\subseteq E.
\]
The reverse order of the limits is what produces the stated order of the
possibly noncommuting product.
\end{proof}


\end{document}